\documentclass[reqno,a4paper, 11pt]{amsart}

\usepackage[a4paper=true,pdfpagelabels]{hyperref}
\usepackage{graphicx}

\usepackage[ansinew]{inputenc}
\usepackage{amsfonts,epsfig}
\usepackage{latexsym}
\usepackage{mathabx}
\usepackage{amsmath}
\usepackage{amssymb}

\usepackage{color}

\newtheorem{theorem}{Theorem}

\newtheorem{corollary}[theorem]{Corollary}
\newtheorem{proposition}[theorem]{Proposition}

\theoremstyle{definition}

\theoremstyle{remark}

\numberwithin{equation}{section}

\newcommand{\D}{\mathbb{D}}
\newcommand{\DD}{\widehat{\mathcal{D}}}
\newcommand{\Dd}{\widecheck{\mathcal{D}}}

\newcommand{\DDD}{\mathcal{D}}

\newcommand{\N}{\mathbb{N}}

\newcommand{\C}{\mathbb{C}}
\newcommand{\Oh}{\mathcal{O}}

\renewcommand{\phi}{\varphi}

\DeclareMathOperator{\Real}{Re}

\newcommand{\T}{\mathbb{T}}

       \def\b{\beta}        \def\g{\gamma}
           
     \def\om{\omega}      
       \def\t{\theta}       
                  \def\z{\zeta}
                  \def\vp{\varphi}
\def\G{\Gamma}

\def\omg{\widehat{\omega}}

\renewcommand{\H}{\mathcal{H}}

\newenvironment{Prf}{\noindent{\emph{Proof of}}}
{\hfill$\Box$ }

\begin{document}

\title[Harmonic conjugates on Bergman spaces induced by doubling weights]{Harmonic conjugates on Bergman spaces induced by doubling weights}


\author{Jos\'e \'Angel Pel\'aez}
\address{Departamento de An\'alisis Matem\'atico, Universidad de M\'alaga, Campus de
Teatinos, 29071 M\'alaga, Spain. Phone number: 0034952131911} \email{japelaez@uma.es} 

\author{Jouni R\"atty\"a}
\address{University of Eastern Finland, P.O.Box 111, 80101 Joensuu, Finland}
\email{jouni.rattya@uef.fi}

\thanks{This research was supported in part by Ministerio de Ciencia Innovación y universidades, Spain, projects
PGC2018-096166-B-100 and MTM2017-90584-REDT; La Junta de Andaluc{\'i}a,
project FQM210; 
Academy of Finland 286877.}

\begin{abstract}
A radial weight $\omega$ belongs to the class $\widehat{\mathcal{D}}$ if there exists $C=C(\omega)\ge 1$ such that
 $\int_r^1 \omega(s)\,ds\le C\int_{\frac{1+r}{2}}^1\omega(s)\,ds$ for all $0\le r<1$. Write $\om\in\widecheck{\mathcal{D}}$ if there exist constants $K=K(\omega)>1$ and $C=C(\omega)>1$ such that $\widehat{\omega}(r)\ge C\widehat{\omega}\left(1-\frac{1-r}{K}\right)$ for all $0\le r<1$. These classes of radial weights arise naturally in the operator theory of Bergman spaces induced by radial weights~\cite{PelaezRattya2019}. 

Classical results by Hardy and Littlewood \cite{HLCrelle32}, and Shields and Williams \cite{ShiWiMich}
show that the weighted Bergman space of harmonic functions is not closed by harmonic conjugation if $\omega\in\widehat{\mathcal{D}}\setminus \widecheck{\mathcal{D}}$
 and $0<p\le 1$. In this paper we establish sharp estimates for the norm of the analytic Bergman space $A^p_\omega$, with $\omega\in\widehat{\mathcal{D}}\setminus \widecheck{\mathcal{D}}$ and $0<p<\infty$, in terms of quantities depending on the real part of the function. It is also shown that these quantities result equivalent norms for certain classes of radial weights.
\end{abstract}

\maketitle

\section{Introduction and main results}

Let $\H(\D)$ and $h(\D)$ denote the spaces of analytic and harmonic functions in the unit disc $\D=\{z\in\C:|z|<1\}$, respectively.
For $0<p\le\infty$, the Hardy space $H^p$ consists of $f\in\H(\D)$ for which
    \begin{equation*}\label{normi}
    \|f\|_{H^p}=\sup_{0<r<1}M_p(r,f)<\infty
    \end{equation*}
where
    $$
    M_p(r,f)=\left (\frac{1}{2\pi}\int_0^{2\pi}
    |f(re^{i\theta})|^p\,d\theta\right )^{\frac{1}{p}},\quad 0<p<\infty,
    $$
and 
    $$
    M_\infty(r,f)=\max_{0\le\theta\le2\pi}|f(re^{i\theta})|.
    $$
The Hardy space $h^p$ of harmonic functions is defined in an analogously manner. For a nonnegative function $\om\in L^1([0,1))$, the extension to $\D$, defined by $\om(z)=\om(|z|)$ for all $z\in\D$, is called a radial weight. For $0<p<\infty$ and such an $\omega$, the Lebesgue space $L^p_\om$ consists of complex-valued measurable functions $f$ on $\D$ such that
    $$
    \|f\|_{L^p_\omega}^p=\int_\D|f(z)|^p\omega(z)\,dA(z)<\infty,
    $$
where $dA(z)=\frac{dx\,dy}{\pi}$ is the normalized area measure on $\D$. The corresponding weighted Bergman spaces are $A^p_\om=L^p_\omega\cap\H(\D)$ and $a^p_\om=L^p_\omega\cap h(\D)$. Throughout this paper we assume $\widehat{\om}(z)=\int_{|z|}^1\om(s)\,ds>0$ for all $z\in\D$, for otherwise $A^p_\om=\H(\D)$ and $a^p_\om=h(\D)$.

A radial weight $\om$ belongs to the class~$\DD$ if the tail integral $\widehat{\om}$ satisfies the doubling property $\widehat{\om}(r)\le C\widehat{\om}(\frac{1+r}{2})$ for some constant $C=C(\om)\ge1$ and for all $0\le r<1$. Write $\om\in\Dd$ if there exist constants $K=K(\om)>1$ and $C=C(\om)>1$ such that $\widehat{\om}(r)\ge C\widehat{\om}\left(1-\frac{1-r}{K}\right)$ for all $0\le r<1$, and set $\DDD=\DD\cap\Dd$. These classes of radial weights arise naturally in the operator theory of weighted Bergman spaces~\cite{PelaezRattya2019}. For instance, the class $\DDD$ describes the radial weights such that the Littlewood-Paley formula 
  \begin{equation}\label{eq:LP1}
	\|f\|^p_{A^p_\om}\asymp |f(0)|^p+\int_{\D}|f'(z)|^p(1-|z|)^p\om(z)\,dA(z),\quad f\in\H(\D),
	\end{equation}
holds~\cite[Theorem~5]{PelaezRattya2019}, and there exists a constant $C=C(\om,p)>0$ such that
	\begin{equation}\label{eq:LP2}
  |f(0)|^p+\int_{\D}|f'(z)|^p(1-|z|)^p\om(z)\,dA(z)\le C\|f\|^p_{A^p_\om},\quad f\in\H(\D),
	\end{equation} 
if and only if $\om\in\DD$ \cite[Theorem~6]{PelaezRattya2019}. These results, together with \cite[Theorems~1 and 3]{PelaezRattya2019} related to bounded Bergman projections on $L^\infty$, show that weights in $\DD\setminus\Dd$ induce in a sense essentially smaller Bergman spaces than the standard radial weights $(1-|z|^2)^\alpha$ with $-1<\alpha<\infty$.  

A classical problem on a space $X$ of harmonic functions in $\D$ consists of studying which properties do a function $u\in X$ and its harmonic conjugate share. In particular, it is of interest to determine whether or not $X$ is closed by conjugation. One of the most celebrated and useful results on this direction is the M.~Riesz theorem which states that there exists a constant $C=C(p)>0$ such that
	\begin{equation}\label{pillu}
	\|f\|_{H^p}\le C\left(|f(0)|+\|\Real f\|_{h^p}\right),\quad f\in\H(\D),
	\end{equation}
whenever $1<p<\infty$~\cite{Duren,Garnett1981}. Despite the fact that this property does not carry over to the range $0<p\le1$, Burkholder, Gundy and Silverstein~\cite{BGS}, see also \cite{Garnett1981}, showed that the equivalence
	\begin{equation}\label{hpmaximal}
	\|f\|_{H^p}\asymp \| (\Real f)^\star\|_{L^p(\T)},\quad f\in\H(\D),
	\end{equation}
is valid for each $0<p<\infty$. Here, and throughout the paper, $u^\star(e^{i\theta})=\sup_{z\in\G(e^{i\theta})}|u(z)|$ and 
	\begin{equation}\label{eq:gammadeu}
  \Gamma(u)
	=\left\{z\in\D:\,|\t-\arg z|<\frac12\left(1-\frac{|z|}{r}\right)\right\},\quad u=re^{i\theta}\in\overline{\D}\setminus\{0\}.
  \end{equation}
As for the Bergman spaces induced by radial weights, by defining 
	$$
	N(u)(\z)=\sup_{z\in\G(\z)}|u(z)|,\quad \z\in\D\setminus\{0\},
  $$ 
and applying \eqref{hpmaximal} to the dilatation $u_r(z)=u(rz)$ with $0\le r<1$ as in the proof of \cite[Lemma~4.4]{PelRat}, one obtains the following result.

\begin{proposition}\label{pr:maxtang}
Let $0<p<\infty$ and let $\om$ be a radial weight. Then there exists a constant $C=C(p)>0$ such that 
	\begin{equation}\label{eq:maxtang}
	\|f\|_{A^p_\om}\le\|N(\Real f)\|_{L^p_\om}\le C\|f\|_{A^p_\om},\quad f\in\H(\D).
	\end{equation}
\end{proposition}

For $1<p<\infty$ and a radial weight $\om$, the inequality 
	\begin{equation}\label{eq:conjapw}
	\|f\|_{A^p_\om}\le C\left( |f(0)|+\| \Real f\|_{a^p_\om} \right),\quad f\in\H(\D),
	\end{equation}
with $C=C(p)>0$ follows from \eqref{pillu}. For $0<p\le1$ we have \eqref{eq:conjapw} for some constant $C=C(\om,p)>0$ if, roughly speaking, $\om$ is sufficiently smooth and $A^p_{\om}$ is large enough \cite[Theorems~5.1 and~7.1]{PavP}.
However, if $A^p_{\om}$ is sufficiently small, then \eqref{eq:conjapw} is no longer true in general for $0<p\le1$. In fact, Hardy and Littlewood~\cite{HLCrelle32}, see also \cite[p.~68]{Duren}, proved that if $0<p\le1$ then each $f\in\H(\D)$ such that $\Real f\in h^p$ satisfies $M_p(r,f)=\Oh\left(\log\frac{e}{1-r}\right)^{\frac1p}$. This implication is sharp in the sense that for $p=\frac{1}{k}$ with $k\in\N$ the function $f(z)=\exp\left(\frac{i(k-1)\pi}{2}\right)(1-z)^{-k}$ satisfies $\Real f\in h^{p}$ and $M_p(r,f)\asymp\left(\log\frac{e}{1-r}\right)^{\frac1p}$ for all $0<r<1$. Consequently, if $\int_0^1\om(s)\log\frac{e}{1-s}\,ds=\infty$, the inequality \eqref{eq:conjapw} fails for $0<p\le1$. Observe that, by Fubini's theorem, $\int_0^1 \om(s)\log\frac{e}{1-s}\,ds$ converges if and only if the positive function~$\widetilde{\om}$, defined by $\widetilde{\om}(r)=\frac{\widehat{\om}(r)}{1-r}$ for all $0\le r<1$, is a radial weight. Unfortunately, $\widetilde{\om}$ being a weight does not guarantee \eqref{eq:conjapw} for $0<p\le1$ if $\om\in\DD\setminus\Dd$. This can be deduced by applying a result of Shields and Williams \cite[Theorem~1']{ShiWiMich}, see the beginning of Section~\ref{Sec:2} for details.  

The main objective of this paper is to search for sharp inequalities of the same type as \eqref{eq:conjapw} but where $\|\Real f\|_{a^p_\om}$on the right has been replaced by only a slightly larger quantity. Of course, the primary interest lies in the case in which $0<p\le1$ and $\om$ induces a relatively small Bergman space, although all the obtained inequalities are valid on the full range $0<p<\infty$. The motivation for the first result comes from Proposition~\ref{pr:maxtang}. In view of \eqref{eq:maxtang} and the obvious inequality 
	$$
	\sup_{0<s< r}M_p(s,\Real f)\le M_p(r,N(\Real f))
	$$ 
it is natural to ask whether or not $M_p(r,\Real f)$ can be replaced by $\sup_{0<s<r}M_p(s,\Real f)$ in \eqref{eq:conjapw}. However, the results in \cite{HLCrelle32} and \cite[Theorem~1']{ShiWiMich} imply at once that this is not the case, and therefore a larger quantity than $\int_0^1\sup_{0<s<r}M^p_p(s,\Real f)\om(r)\,dr$ should be the replacement of $\|\Real f\|^p_{a^p_\om}$ in \eqref{eq:conjapw}. Our first result gives a natural substitute.

\begin{theorem}\label{th:maximalmedias}
Let $0<p<\infty$ and let $\om$ be a radial weight such that $\widetilde{\om}$ is a weight. Then 
	\begin{equation}\label{maxtildeintro}
	\|f\|^p_{A^p_\om}\lesssim\int_0^1\sup_{0<s<r}M^p_p(s,\Real f)\widetilde{\om}(r)\,dr,\quad f\in\H(\D),
	\end{equation}
if and only if $\widetilde{\om}\in\DD$.
\end{theorem}

If $\om\in\DD$ such that also $\widetilde{\om}$ is a weight, then a straightforward calculation shows that
$\widetilde{\om}\in\DD$. The converse implication is false in general as the following result, the proof of which is given in Section~\ref{sec:3}, shows. 

\begin{theorem}\label{th:count1}
There exists a radial weight $\om\not\in\DD$ such that $\widetilde{\om}\in\DD$.
\end{theorem}

The proofs of Theorems~\ref{th:maximalmedias} and \ref{th:count1} use the fact that $\widetilde{\om}\in\DD$ if and only if 
$\widetilde{\om}$ is a weight and there exists a constant $C=C(\om)>0$ such that
	\begin{equation}\label{eq:i1}
	\widehat{\om}(r)\le C\widehat{\widetilde{\om}}(r),\quad0\le r<1.
	\end{equation} 
An integration by parts shows that \eqref{eq:i1} yields
	$$
	\int_0^1 \sup_{0<s<r}M^p_p(s,\Real f)\om(r)\,dr
	\le C\int_0^1 \sup_{0<s<r}M^p_p(s,\Real f)\widetilde{\om}(r)\,dr,\quad f\in\H(\D),
	$$
as expected. Therefore Theorem~\ref{th:maximalmedias} shows that an appropriate substitute for $\|\Real f\|_{a^p_\om}$ in \eqref{eq:conjapw} on the range $0<p\le1$ is the integral on the right hand side of the inequality above.

The proof of Theorem~\ref{th:maximalmedias} is given in Section~\ref{Sec:2}, and it goes roughly speaking as follows. We first observe that \eqref{maxtildeintro} trivially implies $\|f\|_{A^p_\om}\lesssim\|f\|_{A^p_{\widetilde{\om}}}$ for all $f\in\H(\D)$. By testing this with monomials and using a description of $\DD$ in terms of the moments $\om_x=\int_0^1r^x\om(r)\,dr$ of the weight, we deduce $\widetilde\om\in\DD$, which is equivalent to \eqref{eq:i1}. The true work lies in obtaining \eqref{maxtildeintro} on the range $0<p\le1$ from \eqref{eq:i1}. This is achieved by first estimating $\|f\|_{A^p_\om}$ upwards by using the embedding $D^p_{p-1}\subset H^p$ \cite{LUPAMS88} between Dirichlet-type and Hardy spaces, valid for $0<p\le2$, and then passing from the derivative to the real part by using \cite[Lemma~2.2]{Pavcheck}, which states that there exists a constant $C=C(p)>0$ such that $M_p(r,f')\le C(\rho-r)^{-1}\sup_{0<t<\rho}M_p(t,\Real f)$ for all $0\le r<\rho<1$.

Before presenting our next result, we mention an interesting consequence of Theorem~\ref{th:count1} concerning the class $\DD$. Namely, $\DD$ is not closed by multiplication by $(1-|z|)$. Indeed, if $\om$ is the weight constructed in Theorem~\ref{th:count1}, then the weight $\om_{[\beta]}$, defined by $\om_{[\beta]}(z)=\om(z)(1-|z)^\beta$ for all $z\in\D$, does not belong to $\DD$ for any $\beta>0$, see Proposition~\ref{pr:count1} below. These results show that the class $\DD$ of radial weights is much more complex than it seems at first glance. See
\cite{PelSum14,PelaezRattya2019} for an extensive study of $\DD$. 

By our next result, the quantity $\sup_{0<s<r}M^p_p(s,\Real f)$ can be replaced by $M^p_p(r,\Real f)$ in Theorem~\ref{th:maximalmedias} if we pose a stronger hypothesis $\om\in\DD$ instead of $\widetilde\om\in\DD$. 

\begin{theorem}\label{th:W2nuevo}
Let $0<p<\infty$ and $\om\in\DD$. Then there exists a constant $C=C(p,\om)>0$ such that $\|f\|_{A^p_\om}\le C\|\Real f\|_{L^p_{\widetilde{\om}}}$ for all $f\in\H(\D)$.
\end{theorem}

On the most interesting range $0<p\le 1$, the proof of Theorem~\ref{th:W2nuevo} follows the argument used in the corresponding part of the proof of Theorem~\ref{th:maximalmedias}. However, the proof of Theorem~\ref{th:W2nuevo} is more involved because $|\Real f|^p$ is not subharmonic if $0<p<1$. 
Therefore we will use the inequality $M^p_p(r,f')\le C R^{-p-1} \int_{r-R}^{r+R} M_p^p(\rho,\Real f)\,d\rho$, valid for $0<R<r<R+r<1$, to pass from the derivative to the real part. This adds technical difficulties to the proof, which is given in Section~\ref{Sec:2}.

Our next result describes the radial weights for which the quantities appearing in Theorems~\ref{th:maximalmedias} and~\ref{th:W2nuevo} turn out equivalent norms in $A^p_\om$.

\begin{theorem}\label{th:Wintro}
Let $0<p<\infty$ and let $\om$ be a radial weight. Then the following statements are equivalent:
    \begin{enumerate}
     \item[\rm(i)] $\|f\|^p_{A^p_\om}
		\asymp\int_0^1\sup_{0<s< r}M^p_p(s,\Real f)\widetilde{\om}(r)\,dr
		\asymp\int_0^1\sup_{0<s< r}M^p_p(s,\Real f)\om(r)\,dr$ for all $f\in\H(\D)$;
    \item[\rm(ii)] $\|f\|_{A^p_\om}\asymp\|\Real f\|_{L^p_{\widetilde{\om}}}$ for all $f\in\H(\D)$;
    \item[\rm(iii)] $\om\in\DDD$.
    \end{enumerate}
\end{theorem}

The proof of Theorem~\ref{th:Wintro} is given in Section~\ref{Sec:2}, and it strongly uses Theorems~\ref{th:maximalmedias} and~\ref{th:W2nuevo} together with new and known \cite{PelaezRattya2019} descriptions of the class $\DDD$. In particular, it is proved that $\om\in\DDD$ if and only if $\widetilde{\om}\in\DDD$, which is in stark contrast with Theorem~\ref{th:count1}. This implies that the weight $\om$ constructed in Theorem~\ref{th:count1} can not belong to $\DDD$. In fact, one can show that this weight satisfies
	$$
	\frac1{\left(\log\frac1{1-r}\right)^2}\lesssim\widehat{\om}(r)\lesssim\frac{\log\log\log\frac1{1-r}}{\left(\log\frac1{1-r}\right)^2},\quad r\to1^-,
	$$
and therefore induces a weighted Bergman space $A^p_\om$ essentially smaller than any of the standard ones.

As for the question of when the weighted Bergman space is closed by harmonic conjugation, by combining \eqref{eq:LP1} together with \cite[Lemma~3.2]{PavP} and the reasoning in \cite[(3.1)]{PelaezRattya2019}, we obtain 
	\begin{equation}\label{eq:A1}
	\|f\|_{A^1_\om}\le C\left(|f(0)|+\|\Real f\|_{a^1_\om}\right), \quad f\in\H(\D), 
	\end{equation}
provided $\om\in\DDD$. However, we need to pose a stronger hypothesis on $\om$ to extend this result to the range $0<p<1$. 
Recall that a radial weight is regular if $\om(r)\asymp\frac{\widehat{\om}(r)}{1-r}$ for all $0\le r<1$.
Standard weights as well as all the weights in \cite[(4.4)-(4.6)]{Si} are regular. By \cite[Lemma~1.1]{PelRat}, \cite[Lemma~2.1]{PelSum14} and the description of $\Dd$ given in \cite[(2.27)]{PelaezRattya2019}, we deduce that each regular weight belongs to~$\DDD$. Therefore, as a byproduct of Theorem~\ref{th:Wintro} we obtain the following result. 

\begin{corollary}\label{co:final}
Let $0<p<\infty$ and let $\om$ be a regular weight. Then $\|f\|_{A^p_\om}\asymp\|\Real f\|_{L^p_{\omega}}$ for all $f\in\H(\D)$.
\end{corollary} 

Corollary~\ref{co:final} overlaps the earlier result \cite[Theorem~5.1]{PavP} of the same spirit, but none of these results is a consequence of the other one. Namely, Corollary~\ref{co:final} does not assume the continuity of $\om$, but $(5.1)$ in \cite[Theorem~5.1]{PavP}, or the equivalent condition $\sup_{0<r<1}\frac{|\om'(r)|\widehat{\om}(r)}{\om^2(r)}<\infty$, allows the weight to tend to zero much faster than what is allowed by the regularity hypothesis of Corollary~\ref{co:final}; for example the exponential weight given in \cite[Corollary~7.1]{PavP} is not regular.

The last of our main results is Theorem~\ref{pr:count} below. It shows that the validity of \eqref{maxtildeintro} in Theorem~\ref{th:maximalmedias} is equally much related to the regularity of the weight as to its growth/decay. In view of Corollary~\ref{co:final} and the fact that the function $\vp$ in Theorem~\ref{pr:count} may decrease to zero slower than any pregiven rate, it is possible to find $\om$ such that \eqref{maxtildeintro} fails for all $0<p<\infty$, but the space $A^p_{\om}$ lies between two weighted Bergman spaces of which one is just a bit larger than the other and both are closed by conjugation. The proof of Theorem~\ref{pr:count} is given at the end of the paper in Section~\ref{sec:3}.

\begin{theorem}\label{pr:count}
Let $\vp:[0,1]\to[0,\infty)$ be a decreasing function such that $\lim_{r\to1^-}\vp(r)=0$. Then there exists a radial weight $\om=\om_\vp$ such that $\widetilde{\om}\not\in\DD$, but $A^p\subset A^p_{\om}\subset A^p_{\vp}$ for all $0<p<\infty$. If $\vp$ satisfies $-\vp'(t)/\vp(t)\lesssim1/(1-t)$ for all $0\le t<1$, then $\vp$ is regular.
\end{theorem}

The letter $C=C(\cdot)$ will denote an absolute constant whose value depends on the parameters indicated
in the parenthesis, and may change from one occurrence to another.
We will use the notation $a\lesssim b$ if there exists a constant
$C=C(\cdot)>0$ such that $a\le Cb$, and $a\gtrsim b$ is understood
in an analogous manner. In particular, if $a\lesssim b$ and
$a\gtrsim b$, then we write $a\asymp b$ and say that $a$ and $b$ are comparable. This notation has already been used above in the introduction.

\section{Proofs of Theorems~\ref{th:maximalmedias}, \ref{th:W2nuevo} and \ref{th:Wintro}}\label{Sec:2}

Before presenting the proofs, let us see that $a^1_\om$ 
is not closed by harmonic conjugation if $\om\in\DD\setminus\Dd$ even if $\widetilde{\om}$ is a weight.
Let $\psi$ be a positive increasing function on $[0,\infty)$ such that $\psi(x)x^{-\alpha}$ is essentially decreasing on $[\frac12,\infty)$ for some $\alpha>0$ large enough. Shields and Williams \cite[Theorem 1']{ShiWiMich}  showed that if $\psi$ is in addition sufficiently smooth, for example convex or concave, then there exists an $f\in\H(\D)$ such that $M_1(r,\Real f)\lesssim\psi\left(\frac1{1-r}\right)$, but 	
	$$
	M_1(r,f)\asymp\int_{\frac12}^{1/(1-r)}\psi(t)\,\frac{dt}{t},\quad r\to1^-.
	$$
Let us observe that 
	$$
	\limsup_{r\to1^-}\frac{\int_{\frac12}^{1/(1-r)}\psi(t)\frac{dt}{t}}{\psi\left(\frac1{1-r}\right)}=\infty
	$$
if there does not exist $\b>0$ such that $\psi(x)x^{-\b}$ is essentially increasing on $[\frac12,\infty)$. This is the case for example, if $\psi(x)=\log(x+2)$~\cite[Lemma 2]{ShiWiMich}. Let now $\om$ be a radial weight such that $\widetilde{\om}\in\DD$  and therefore $\widehat{\om}\lesssim\widehat{\widetilde{\om}}$ on $[0,1)$. Assume that $\widehat{\widetilde{\om}}$ is smooth enough meaning that there exists a $\psi=\psi_\om$ for which
	$$
	\psi\left(x\right)\asymp\left(\widehat{\widetilde{\om}}\left(1-\frac1x\right)\right)^{-1},\quad 0\le x<\infty.
	$$
Then
	$$
	\|f\|_{A^1_{\om}}
	\gtrsim\int_0^1\left(\int_1^{\frac1{1-r}}\frac{\psi(t)}{t}\,dt\right)\om(r)r\,dr
	\gtrsim\int_{\frac12}^1\psi\left(\frac1{1-r}\right)\widetilde{\om}(r)\,dr
	=-\lim_{r\to1^-}\log\widehat{\widetilde{\om}}(r)-c=\infty,
	$$
but
	$$
	\|\Real f\|_{A^1_\om}
	\lesssim\int_0^1\psi\left(\frac1{1-t}\right)\om(s)\,ds=c+\int_0^1\left(\frac{\widehat{\om}(t)}{\widehat{\widetilde{\om}}(t)}\right)^2\frac{dt}{1-t}.
	$$
This last integral might very well converge as is seen by considering the rapidly increasing weight $v_\alpha(r)=(1-r)^{-1}\left(\log\frac{e}{1-r}\right)^{-\alpha}$, where $2<\alpha<\infty$. Namely, $\widehat{v_\alpha}(t)/\widehat{\widetilde{v_\alpha}}(t)\asymp\left(\log\frac{e}{1-r}\right)^{-1}$ for all $0\le r<1$. In this case the natural choice for the smooth $\psi$ is $\psi(x)=(\log(x+2))^{\alpha-2}$, which satisfies the hypotheses of the result by Shields and Williams.

Now we proceed to prove following result which contains
Theorem~\ref{th:maximalmedias}.

\begin{theorem}\label{th:1}
Let $0<p<\infty$ and let $\om$ be a radial weight such that $\widetilde{\om}$ is a weight. Then the following statements are equivalent:
    \begin{enumerate}
    \item[\rm(i)] There exists a constant $C=C(\om,p)>0$ such that 
				\begin{equation}\label{maxtilde}
				\|f\|^p_{A^p_\om}\le C \int_0^1 \sup_{0<s< r}M^p_p(s,\Real f)\widetilde{\om}(r)\,dr,
				\end{equation}
for all $f\in\H(\D)$;
    \item[\rm(ii)] There exists a constant $C=C(\om,p)>0$ such that $\|f\|_{A^p_\om}\le C\|f\|_{A^p_{\widetilde{\om}}}$ for all $f\in\H(\D)$;
    \item[\rm(iii)]  There exists a constant $C=C(\om)>0$ such that $\widehat{\om}(r)\le C\widehat{\widetilde{\om}}(r)$ for all $0\le r<1$;
    \item[\rm(iv)]   $\widetilde{\om}\in\DD$.
    \end{enumerate}
\end{theorem}

\begin{proof}
It is clear that (i) implies (ii). Assume now (ii). It is known that a radial weight $\nu$ belongs to $\DD$ if and only if for some (equivalently for each) $\b>0$, there exists a constant $C=C(\om,\b)>0$ such that
    \begin{equation}\label{Corollary:D-hatn}
    x^\b(\nu_{[\b]})_x\le C\nu_x,\quad 0\le x<\infty,
    \end{equation}
see \cite[(3.3)]{PelaezRattya2019} for details. By testing (ii) with the monomials $m_n(z)=z^n$, we deduce $\om_{np+1}\le C\widetilde{\om}_{np+1}$, from which an integration by parts on the left yields $(np+1)(\widetilde{\om}_{[1]})_{np}\le C\widetilde{\om}_{np+1}$ for all $n\in\N\cup\{0\}$. By choosing $np\le x<np+1$ we obtain $x(\widetilde{\om}_{[1]})_{x}\le C\widetilde{\om}_{x}$ for all $0\le x<\infty$. Hence $\widetilde{\om}\in\DD$ by \eqref{Corollary:D-hatn}, and thus (iv) is satisfied.

Assume next (iv). Then there exists a constant $C=C(\widetilde{\om})>1$ such that
	$$
	\log2\widehat{\om}\left(\frac{1+r}{2}\right)
	\le\int_{r}^{\frac{1+r}{2}}\widetilde{\om}(s)\,ds
	\le C\int_{\frac{1+r}{2}}^1\widetilde{\om}(s)\,ds,\quad 0\le r<1.
	$$
This yields (iii).

To complete the proof it remains to show that (i) follows from (iii). Let $f\in A^p_{\widetilde{\om}}$ with $f(0)=0$. Then (iii) yields
    $$
    M^p_p(r,f)\widehat{\om}(r)\lesssim M^p_p(r,f)\widehat{\widetilde{\om}}(r)\le\frac1r\int_{\D\setminus D(0,r)}|f(z)|^p\widetilde{\om}(z)\,dA(z)\to0,\quad r\to1^-.
    $$
Hence an integration by parts together with (iii) gives
  \begin{equation}\label{8n}
  \begin{split}
  \int_{0}^{1}M_p^p(r,f)\om(r)\,dr
  &=\int_{0}^{1} \frac{\partial}{\partial r}M_p^p(r,f)\omg(r)\,dr
  \lesssim\int_{0}^{1} \frac{\partial}{\partial r}M_p^p(r,f)\widehat{\widetilde{\om}}(r)\,dr\\
  &=\int_{0}^{1}M_p^p(r,f)\widetilde{\om}(r)\,dr,
  \end{split}
  \end{equation}
from which (ii) follows by standard arguments. If $p>1$, then the inequality in (i) follows from (ii) and M. Riesz Theorem~\cite{Duren}. Thus we have deduced (i) from (iii) in the case $p>1$ by passing through (ii).

To deal with the case $0<p\le1$, we first show that (iii) implies (iv). Let $\g>C$, where $C=C(\om)>0$ is that of (iii). Then Fubini's theorem yields
	\begin{equation*}
	\begin{split}
	\int_0^t\left(\frac{1-t}{1-s}\right)^\g\widetilde{\om}(s)\,ds
	&\le C\int_0^t\left(\frac{1-t}{1-s}\right)^\g\frac{\widehat{\widetilde{\om}}(s)}{1-s}\,ds\\
	&=C\widehat{\widetilde{\om}}(t)(1-t)^\g\int_0^t\frac{ds}{(1-s)^{\g+1}}
	+C\int_0^t\left(\frac{1-t}{1-s}\right)^\g\left(\int_s^t\widetilde{\om}(x)\,dx\right)\frac{ds}{1-s}\\
	&\le\frac{C}{\g}\widehat{\widetilde{\om}}(t)+\frac{C}{\g}\int_0^t\left(\frac{1-t}{1-x}\right)^\g\widetilde{\om}(x)\,dx,
	\end{split}
	\end{equation*}
and it follows that
	$$
	\int_0^t\left(\frac{1-t}{1-s}\right)^\g\widetilde{\om}(s)\,ds\le\frac{C}{\g-C}\widehat{\widetilde{\om}}(t),\quad 0\le t<1.
	$$
Therefore $\widetilde\om\in\DD$ by \cite[Lemma~2.1]{PelSum14}.

Assume now, without loss of generality, that $f(0)=0$. Denote $h(r)=\sup_{0<s<r}M^p_p(s,\Real f)$ for short, and assume $h\in L^p_{\widetilde{\om}}$, otherwise in (i) there is nothing to prove. The Dirichlet-type space $D^p_{p-1}$ consists of those $f$ such that $f'\in A^p_{p-1}$, and it satisfies the well-known embedding
	\begin{equation}\label{eq:DpHp}
	D^p_{p-1}\subset H^p,\quad 0<p\le 2,
	\end{equation}
by \cite{LUPAMS88}. Moreover, by \cite[Lemma~2.2]{Pavcheck}, there exists a constant $C=C(p)>0$ such that
	\begin{equation}\label{eq:derivativemaximalu}
	M_p(r,f')\le C(\rho-r)^{-1}\sup_{0<t<\rho}M_p(t,\Real f),\quad 0\le r<\rho<1.
	\end{equation}
By combining these facts and using Fubini's theorem we deduce
    \begin{equation}
    \begin{split}\label{eq:m1}
    \|f\|^p_{A^p_\om}
    &=\int_{0}^1\|f_r\|_{H^p}^p\om(r)r\,dr
    \lesssim\int_{0}^1\left(\int_0^1M_p^p(rt,f')(1-t)^{p-1}\,dt\right)r^{p+1}\om(r)\,dr\\
    &\lesssim\int_{0}^1\left(\int_0^1h\left( \frac{1+tr}{2}\right)\frac{(1-t)^{p-1}}{(1-tr)^p}\,dt\right)r^{p+1}\om(r)\,dr
    =I_1(r)+I_2(r),
    \end{split}
    \end{equation}
where 
		$$   
		I_1(r)=\int_{0}^1\left(\int_0^rh\left(\frac{1+rt}{2}\right)\frac{(1-t)^{p-1}}{(1-tr)^p}\,dt\right)r^{p+1}\om(r)\,dr
		$$
and
		$$
		I_2(r)=\int_{0}^1\left(\int_r^1h\left(\frac{1+tr}{2}\right)\frac{(1-t)^{p-1}}{(1-tr)^p}\,dt\right)r^{p+1}\om(r)\,dr.
		$$
Fubini's theorem and $\widetilde{\om}\in\DD$ imply
		\begin{equation}
    \begin{split}\label{eq:m2}
    I_1(r)
		&\le\int_{0}^1\left(\int_0^rh\left(\frac{1+t}{2}\right)\frac{dt}{1-t}\right)\om(r)\,dr
		=\int_{0}^1h\left(\frac{1+t}{2}\right)\widetilde{\om}(t)\,dt\\ 
		&\lesssim\int_{0}^1h\left(\frac{1+t}{2}\right)\widetilde{\om}\left(\frac{1+t}{2}\right)\,dt
		\lesssim\int_{0}^1h\left(s\right)\widetilde{\om}(s)\,ds.
    \end{split}
    \end{equation}
Moreover, by using (iii) we deduce
		$$ 
		h\left(\frac{1+t}{2}\right)\widehat{\om}(t)
		\lesssim h\left(\frac{1+t}{2}\right)\widehat{\widetilde{\om}}(t)
		\asymp h\left(\frac{1+t}{2}\right)\widehat{\widetilde{\om}}\left(\frac{1+t}{2}\right)
		\le\int_{\frac{1+t}{2}}^1h(r)\widetilde{\om}(r)\,dr\to 0,\quad t\to 1^-.
		$$        
Therefore an integration by parts, (iii) and the fact that $\widetilde{\om}\in\DD$ yield 
		\begin{equation}
		\begin{split}\label{eq:m3}
    I_2(r)
		&\le\int_{0}^1h\left(\frac{1+r}{2}\right)\left(\int_r^1(1-t)^{p-1}\,dt\right)\frac{\om(r)}{(1-r)^p}\,dr
		\asymp\int_{0}^1h\left(\frac{1+r}{2}\right)\om(r)\,dr\\ 
		&=h\left(\frac{1}{2}\right)\widehat{\om}(0)
		+\int_{0}^1\frac{d}{dt} h\left(\frac{1+t}{2}\right)\widehat{\om}(t)\,dt\\ 
		&\lesssim h\left(\frac{1}{2}\right)\widehat{\widetilde{\om}}(0)
		+\int_{0}^1\frac{d}{dt}h\left(\frac{1+t}{2}\right)\widehat{\widetilde{\om}}(t)\,dt
    \lesssim\int_{0}^1h\left(t\right)\widetilde{\om}(t)\,dt.
     \end{split}
    \end{equation}
By combining \eqref{eq:m1}, \eqref{eq:m2} and \eqref{eq:m3}, we finally deduce \eqref{maxtilde}. This finishes the proof of the theorem.
\end{proof}

\begin{Prf}{\em{ Theorem~\ref{th:W2nuevo}. }}
If $\om\in\DD$, then
    $$
    \widehat{\widetilde{\om}}(r)\lesssim\int_r^1\widetilde\om\left(\frac{1+s}{2}\right)\,ds=2\widehat{\widetilde{\om}}\left(\frac{1+r}{2}\right),\quad0\le r<1,
    $$
and hence $\widetilde{\om}\in\DD$. Therefore $\|f\|_{A^p_\om}\lesssim\|f\|_{A^p_{\widetilde{\om}}}$ for all $f\in\H(\D)$, by Theorem~\ref{th:1}. This together M.~Riesz theorem \cite[Theorem~4.1]{Duren} gives the claim of the theorem for $p>1$.

To deal with the case $0<p\le1$, assume without loss of generality that $f(0)=0$. 
We will employ an argument similar to that used in the proof Theorem~\ref{th:1}. However, the present situation yields more involved considerations because the inequality \eqref{eq:derivativemaximalu} is replaced by the estimate
	\begin{equation}\label{eq:derivativeu}
	M^p_p(r,f')\le C R^{-p-1} \int_{r-R}^{r+R} M_p^p(\rho,\Real f)\,d\rho,\quad 0<R<r<R+r<1,\quad f\in\H(\D),
	\end{equation}
which can be found in \cite[Lemma~5.3]{PavP}, and where $C=C(p)>0$ is a constant. 

Let $r_0=1/\sqrt{3}$ and $t_0\in\left(\frac{2+r_0}{3},1\right)$. Apply \eqref{eq:DpHp}, \eqref{eq:derivativeu} and Fubini's theorem to obtain 
    \begin{equation*}
    \begin{split}
    \|f\|^p_{A^p_\om}
    &\asymp\int_{r_0}^1M^p_p(r,f)\om(r)r\,dr
    \lesssim\int_{r_0}^1\left(\int_0^1M_p^p(rt,f')(1-t)^{p-1}\,dt\right)\om(r)r^{p+1}\,dr\\
    &\asymp\int_{r_0}^1\left(\int_{t_0}^1M_p^p(rt,f')(1-t)^{p-1}\,dt\right)\om(r)\,dr\\
    &=\int_{t_0}^1\left( \int_{r_0}^1M_p^p(rt,f')\om(r)\,dr\right)(1-t)^{p-1}\,dt\\
    &\lesssim\int_{t_0}^1\left(\int_{r_0}^1
    \left(\int_{\frac{3tr-1}{2}}^{\frac{1+tr}{2}}M_p^p(s,\Real f)\,ds\right)\frac{\om(r)}{(1-tr)^{p+1}}dr\right)(1-t)^{p-1}\,dt\\
    &=I_1(f)+I_2(f)+I_3(f),
    \end{split}
    \end{equation*}
where
    \begin{equation*}
    \begin{split}
    I_1(f)&=\int_{t_0}^1\left(\int_{\frac{3tr_0-1}{2}}^{\frac{1+tr_0}{2}}M_p^p(s,\Real f)\left(\int_{r_0}^{\frac{2s+1}{3t}} \frac{\om(r)}{(1-tr)^{p+1}}\,dr\right)ds\right)(1-t)^{p-1}\,dt,\\
    I_2(f)&=\int_{t_0}^1\left(\int_{\frac{1+tr_0}{2}}^{\frac{3t-1}{2}}M_p^p(s,\Real f)\left(\int_{\frac{2s-1}{t}}^{\frac{2s+1}{3t}}
    \frac{\om(r)}{(1-tr)^{p+1}}\,dr\right)ds\right)(1-t)^{p-1}\,dt,\\
    I_3(f)&=\int_{t_0}^1\left(\int_{\frac{3t-1}{2}}^{\frac{1+t}{2}}M_p^p(s,\Real f)\left(\int_{\frac{2s-1}{t}}^{1}
    \frac{\om(r)}{(1-tr)^{p+1}}\,dr\right)ds\right)(1-t)^{p-1}\,dt.
    \end{split}
    \end{equation*}
Since $\frac{1+tr_0}{2}\le\frac{1+r_0}{2}<1$ and $\frac{2s+1}{3t}\le\frac{2+r_0}{3t}\le\frac{2+r_0}{3t_0}<1$ for $s\le\frac{1+r_0}{2}$ and $t_0\le t<1$, we have
    \begin{equation*}
    \begin{split}
    I_1(f)
		&\le\int_{t_0}^1(1-t)^{p-1}\,dt\left(\int_{0}^{\frac{1+r_0}{2}}M_p^p(s,\Real f)\,ds\right)\left(\int_{0}^{\frac{2+r_0}{3t_0}} \frac{\om(r)}{(1-r)^{p+1}}\,dr\right)\\
    &\lesssim\int_0^{\frac{1+r_0}{2}}M_p^p(s,\Real f)\widetilde{\om}(s)s\,ds\le\|\Real f\|^p_{L^p_{\widetilde{\om}}}.
    \end{split}
    \end{equation*}
To estimate $I_2(f)$ and $I_3(f)$, observe first that in both cases $s\le\frac{1+t}{2}$, that is, $\frac{2s-1}{t}\le1$, and $\frac{2s-1}{t}\ge\min\{r_0,3-\frac2{t_0}\}>0$. By \cite[Lemma~2.1(ii)]{PelSum14}, there exists $\b=\b(\om)>0$ such that
    $$
    \omg\left( \frac{2s-1}{t}\right)\lesssim\omg(s)\frac{1}{t^\b}\left(\frac{t+1-2s}{1-s} \right)^\b\le \frac{2^\b}{t_0^\b}\omg(s),\quad t\ge t_0,\quad \frac{2s-1}{t}<s.
    $$
It follows that for $s$ and $t$ on the ranges of values appearing in $I_2(f)$ and $I_3(f)$ we have $\omg\left(\frac{2s-1}{t}\right)\lesssim\omg(s)$. Hence Fubini's theorem yields
  \begin{equation*}
  \begin{split}
  I_2(f)
	&\asymp\int_{t_0}^1\left(\int_{\frac{1+tr_0}{2}}^{\frac{3t-1}{2}}\frac{M_p^p(s,\Real f)}{(1-s)^{p+1}}
  \left(\omg\left(\frac{2s-1}{t}\right)-\omg\left(\frac{2s+1}{3t}\right)\right)ds\right)(1-t)^{p-1}\,dt\\
  &\lesssim\int_{t_0}^1\left(\int_{0}^{\frac{3t-1}{2}}\frac{M_p^p(s,\Real f)}{(1-s)^{p+1}}\omg(s)\,ds\right)(1-t)^{p-1}\,dt\\
  &\le\int_0^1M_p^p(s,\Real f)\left(\int_{\frac{2s+1}{3}}^1(1-t)^{p-1}\,dt\right)\frac{\omg(s)}{(1-s)^{p+1}}ds\\
  &\asymp\int_0^1M_p^p(s,\Real f)\widetilde{\om}(s)\,ds
  \lesssim\|\Real f\|^p_{L^p_{\widetilde{\om}}}.
  \end{split}
  \end{equation*}
As for the remaining term, by using $\widehat{\om}\left(\frac{2s-1}{t}\right)\lesssim\widehat{\om}(s)$ and Fubini's theorem,
    \begin{equation*}
    \begin{split}
    I_3(f)
    &\le\int_{t_0}^1\left(\int_{\frac{3t-1}{2}}^{\frac{1+t}{2}} M_p^p(s,\Real f)\omg\left(\frac{2s-1}{t} \right)\,ds\right)\frac{dt}{(1-t)^{2}}\\
    &\lesssim\int_{t_0}^1\left(\int_{\frac{3t-1}{2}}^{1} M_p^p(s,\Real f)\omg\left(s\right)\,ds\right)\frac{dt}{(1-t)^{2}}\\
    &=\int_{\frac{3t_0-1}{2}}^{1}M_p^p(s,\Real f)\widehat\om\left(s\right)\left(\int_{t_0}^{\frac{2s+1}{3}}\frac{dt}{(1-t)^2}\right)ds
    \lesssim\|\Real f\|^p_{L^p_{\widetilde{\om}}}.
   \end{split}\end{equation*}
By combining the estimates for $I_1(f)$, $I_2(f)$ and $I_3(f)$ we deduce the claim.
\end{Prf}

\medskip

Theorem~\ref{th:Wintro} is obtained in the following result.

\begin{theorem}\label{th:W}
Let $0<p<\infty$ and let $\om$ be a radial weight. Then the following statements are equivalent:
    \begin{enumerate}
     \item[\rm(i)] $\|f\|^p_{A^p_\om}\asymp\int_0^1 \sup_{0<s< r}M^p_p(s,\Real f)\,  \widetilde{\om}(r)\,dr\asymp \int_0^1 \sup_{0<s< r}M^p_p(s,\Real f)\, \om(r)\,dr$ for all $f\in\H(\D)$;
    \item[\rm(ii)] $\|f\|_{A^p_\om}\asymp\|\Real f\|_{L^p_{\widetilde{\om}}}$ for all $f\in\H(\D)$;
     \item[\rm(iii)] $\|f\|_{A^p_\om}\asymp\|f\|_{A^p_{\widetilde{\om}}}$ for all $f\in\H(\D)$;
    \item[\rm(iv)] $\widetilde{\om}\in\DDD$;
    \item[\rm(v)] $\om\in\DDD$.
    \end{enumerate}
\end{theorem}

\begin{proof}
We first show that (iii), (iv) and (v) are equivalent. Assume (v). Since $\om\in\DDD\subset\Dd$, there exist $C=C(\om)>0$ and $\b=\b(\om)>0$ such that
    \begin{equation}\label{6}
    \begin{split}
    \widehat{\om}(r)\le C\left(\frac{1-r}{1-t}\right)^{\b}\widehat{\om}(t),\quad 0\le t\le r<1,
    \end{split}
    \end{equation}
by \cite[(2.27)]{PelaezRattya2019}. Therefore $\widehat{\om}(r)\lesssim(1-r)^\b$ for all $0\le r<1$, and hence $\widetilde{\om}$ is a weight. Moreover, by the definition of $\Dd$, there exist $K=K(\om)>1$ and $C=C(\om)>1$ such that
    $$
    \widehat{\widetilde{\om}}(r)
    \ge C\int_r^1\frac{\omg\left(1-\frac{1-s}{K}\right)}{1-s}\,ds
    =C\widehat{\widetilde{\om}}\left(1-\frac{1-r}{K}\right),\quad 0\le r<1,
    $$
and hence
    \begin{equation*}
    \begin{split}
    \widehat{\widetilde{\om}}(r)
    &\le\int_r^{1-\frac{1-r}{K}}\widetilde{\om}(s)\,ds+\frac{\widehat{\widetilde{\om}}(r)}{C}
    \le\omg(r)\log K+\frac{\widehat{\widetilde{\om}}(r)}{C},\quad 0\le r<1.
    \end{split}
    \end{equation*}
It follows that $\widehat{\widetilde{\om}}(r)\lesssim\omg(r)$ for all $0\le r<1$. On the other hand, since $\om\in\DDD\subset\DD$, we have
	$$
	\widehat{\om}(r)
	\le C \widehat{\om}\left(\frac{1+r}{2}\right)
	\le\frac{C}{\log 2}\int_{r}^{\frac{1+r}{2}}\frac{\widehat{\om}(s)}{1-s}\,ds
	\le\frac{C}{\log 2}\widehat{\widetilde{\om}}(r),\quad 0\le r<1.
	$$
Consequently, 	
		\begin{equation}\label{eq:gorros}
		\widehat{\widetilde{\om}}(r)\asymp\omg(r), \quad 0\le r<1.
		\end{equation}
Arguing now as in \eqref{8n} we deduce (iii). An alternative way to deduce (iii) from (v) is to first observe that $\om\in\DDD$ implies $\widetilde\om\in\DDD$, and then use Carleson measures \cite[Theorem~1]{PelRatEmb} together with \eqref{eq:gorros}. We omit the details of this alternative approach.

Assume next (iii). Then $\widetilde\om$ must be a weight because $\om$ is.
By testing (iii) with the monomials~$m_n$, we deduce $\om_{np+1}\asymp\widetilde{\om}_{np+1}$, from which an integration by parts on the left yields $(np+1)(\widetilde{\om}_{[1]})_{np}\asymp \widetilde{\om}_{np+1}$ for all $n\in\N\cup\{0\}$. By choosing $np\le x<np+1$ we obtain $x(\widetilde{\om}_{[1]})_{x}\asymp \widetilde{\om}_{x}$ for all $0\le x<\infty$. By using now \cite[(1.2), (1.3), Theorem~3]{PelaezRattya2019}, we deduce $\widetilde{\om}\in\DDD$, and thus (iv) holds.

Assume (iv).  Since $\widetilde\om\in\DDD$ %=\M\cap\DD=\Dd\cap\DD$
 by \cite[Theorem~3]{PelaezRattya2019}, \cite[(1.2) and (1.3)]{PelaezRattya2019} yield $\left(\widetilde\om_{[\b]}\right)_x\asymp x^{-\b}\widetilde\om_x$ for all $x\ge1$ and $\b>0$. Therefore \cite[Lemma~2.1 (ix)]{PelSum14} yields
	\begin{equation*}
	\begin{split}
	\left(\widetilde\om_{[\b]}\right)_x
	\asymp\frac{\widetilde\om_x}{x^{\b}}
	\lesssim\frac{\widetilde\om_{2x}}{(2x)^{\b}}
	\asymp\left(\widetilde\om_{[\b]}\right)_{2x},\quad x\ge1,
	\end{split}
	\end{equation*}
and thus $\widetilde\om_{[\b]}\in\DD$ by \cite[Lemma~2.1]{PelSum14}. In particular, by choosing $\b=1$, we deduce $\widehat{\om}\in\DD$. Since $(x+1)\om_x=\widehat{\om}_{x+1}$ for $x>0$, \cite[Lemma~2.1(ix)]{PelSum14} implies $\om\in\DD$. Further, since $\widetilde\om\in\Dd$, there exist $K=K(\om)>1$ and $C=C(\om)>1$ such that
	$$
	\int_{r}^{1-\frac{1-r}{K}}\widetilde\om(s)\,ds\ge (C-1)\int_{1-\frac{1-r}{K}}^1\widetilde\om(s)\,ds,\quad 0\le r<1.
	$$
Therefore, for $M>K$, we have
	\begin{equation*}
	\begin{split}
	\widehat{\om}(r)\log K 
	&\ge\int_{r}^{1-\frac{1-r}{K}}\widetilde\om(s)\,ds
	\ge(C-1)\int_{1-\frac{1-r}{K}}^1\widetilde\om(s)\,ds
	\ge(C-1)\int_{1-\frac{1-r}{K}}^{1-\frac{1-r}{M}}\widetilde\om(s)\,ds\\
	&\ge(C-1)\widehat{\om}\left(1-\frac{1-r}{M}\right)\log \frac{M}{K},\quad 0\le r<1.
	\end{split}
	\end{equation*}
By choosing $M$ large enough such that $(C-1)\log\frac{M}{K}>\log K$, it follows that $\om\in\Dd$. Thus $\om\in\DDD$, that is, (v) holds.
Therefore we have shown that (iii), (iv) and (v) are equivalent.

Assume now (ii). Then $\widetilde\om$ must be a weight because $\om$ is. Moreover, by testing (ii) with the monomials $m_n$, we obtain
    \begin{equation*}
    \begin{split}
    \om_{np+1}\asymp\|\Real m_n\|^p_{L^p_{\widetilde{\om}}}
    &\asymp\int_{0}^1 r^{np+1}\widetilde{\om}(r)\left(\int_{-\pi}^\pi|\cos(n\theta)|^p\,d\theta\right)dr
    \asymp\widetilde{\om}_{np+1},\quad n\in\N\cup\{0\}.
    \end{split}
    \end{equation*}
By arguing as in the proof above where we showed that (v) follows from (iii), we deduce $\widetilde{\om}\in\DDD$. 

Conversely, if $\widetilde{\om}\in\DDD$, then also $\om\in\DDD$, and (iii) is satisfied by the first part of the proof. Therefore we may use Theorem~\ref{th:W2nuevo} and (iii) to obtain 
    $$
    \|f\|_{A^p_\om}
    \lesssim\|\Real f\|_{L^p_{\widetilde{\om}}}
    \le\|f\|_{A^p_{\widetilde{\om}}}
    \asymp\|f\|_{A^p_{\om}}, \quad f\in \H(\D).
    $$
Thus (ii) is satisfied. Hence we have shown that (ii)-(v) are equivalent.

It remains to associate (i) with the other conditions. This is achieved by using the results we have already proved along with techniques similar to those used in the proofs above. Therefore we only indicate the proofs and omit the details. If (i) is satisfied, then by testing with the monomials $m_n$ we deduce $\widetilde{\om}\in\DDD$. Conversely, if $\widetilde{\om}\in\DDD$, then Theorem~\ref{th:maximalmedias} and an integration by parts similar to that in \eqref{8n} together with \eqref{eq:gorros} give 
		\begin{equation*}
		\begin{split}
		\|f\|^p_{A^p_\om}
		\lesssim\int_0^1\sup_{0<s<r}M^p_p(s,\Real f)\widetilde{\om}(r)\,dr
		\asymp\int_0^1\sup_{0<s<r}M^p_p(s,\Real f){\om}(r)\,dr
		\lesssim\|f\|^p_{A^p_\om},\quad f\in\H(\D).
		\end{split}
		\end{equation*}
This finishes the proof of the theorem.
\end{proof}

\section{Counterexamples}\label{sec:3}

In this section we will prove Theorems~\ref{th:count1} and~\ref{pr:count} by constructing radial weights with desired properties. The first construction shows that there exist weights $\om\not\in\DD$ such that $\widetilde{\om}\in\DD$. The second one illustrates the fact that even if $\widetilde\om$ is differentiable almost everywhere whenever it is a weight, the validity of the inequality \eqref{maxtilde} is equally much related to the regularity of the weight $\om$ as to its growth/decay.

Before proving the above mentioned two results, we will show, by using Theorem~\ref{th:count1}, that the class $\DD$ is not closed by multiplication by $(1-|z|)^\beta$ for any $\beta>0$. This indicates that despite of its innocent definition, the class $\DD$ has in a sense complex nature. 

\begin{proposition}\label{pr:count1}
The implication $\om\in\DD\Rightarrow\om_{[\b]}\in\DD$ is in general false for each $\b>0$.
\end{proposition}

\begin{proof}
First, let us observe that for each $\nu\in\DD$ and an increasing function $\psi:[0,1)\to(0,\infty)$, with $\psi(z)=\psi(|z|)$ for all $z\in\D$, such that $\psi\nu$ is a weight, we have $\psi\nu\in\DD$. Indeed, since $\nu\in\DD$ there exists a constant $C=C(\nu)>1$ such that $\widehat{\nu}(r)\le C\widehat{\nu}\left(\frac{1+r}{2}\right)$ for all $0\le r<1$. This is equivalent to
	\begin{equation*}
	\begin{split}
	\int_r^{\frac{1+r}{2}}\nu(s)\frac{\psi(s)}{\psi(s)}\,ds\le\left(C-1\right)\int_{\frac{1+r}{2}}^1\nu(s)\frac{\psi(s)}{\psi(s)}\,ds,
	\end{split}
	\end{equation*}
which implies
	\begin{equation*}
	\begin{split}
	\frac1{\psi\left(\frac{1+r}{2}\right)}\int_r^{\frac{1+r}{2}}\nu(s)\psi(s)\,ds
	\le\frac{C-1}{\psi\left(\frac{1+r}{2}\right)}\int_{\frac{1+r}{2}}^1\nu(s)\psi(s)\,ds,
	\end{split}
	\end{equation*}
that is, $\widehat{\psi\nu}(r)\le C\widehat{\psi\nu}\left(\frac{1+r}{2}\right)$ for all $0\le r<1$. Thus $\psi\nu\in\DD$.

Next, assume on the contrary to the statement that $\nu\in\DD$ implies $\nu_{[\b]}\in\DD$ for some $\b>0$ and all $\nu\in\DD$. Pick $n\in\N$ such that $n\b\ge1$, and let $\om$ be the weight of the statement of Theorem~\ref{th:count1}. Then $n$ applications of the antithesis to $\nu=\widetilde\om$ imply $\widehat\om_{[n\b-1]}\in\DD$. Further, the previous obervation with $\psi(z)=(1-|z|)^{1-n\b}$ gives $\widehat{\om}\in\DD$, which is equivalent to $\om\in\DD$ by Fubini's theorem and \cite[Lemma~2.1]{PelSum14}(ix).
This is a contradiction.
\end{proof}

\begin{Prf}{\em{Theorem~\ref{th:count1}.}}
Let $t_n=1-e^{-n}$ for all $n\in\N$, $\vp:\N\setminus\{1\}\to(0,1/2)$ such that $\vp\in\ell^1$, and define $s_n\in(t_n,t_{n+1})$ by $\frac{1-s_n}{1-t_{n+1}}=e^{\vp(n)}$ for all $n\in\N$. Further, let $\om(s)=\sum_{n=2}^\infty\frac{\chi_{[s_n,t_{n+1}]}(s)}{1-s}$ for all $0\le s<1$. The function $\vp$ will be appropriately fixed later. Then
	$$
	\widehat{\om}(0)
	=\sum_{n=2}^\infty\int_{s_n}^{t_{n+1}}\frac{ds}{1-s}
	=\sum_{n=2}^\infty\log\frac{1-s_n}{1-t_{n+1}}
	=\sum_{n=2}^\infty\vp(n)<\infty,
	$$
and thus $\om$ is a radial weight. Moreover,
	\begin{equation}\label{eq:tilde}
	\begin{split}
	\widehat{\om}(r)=\left\{
        \begin{array}{ll}
        \sum_{k=n}^\infty\vp(k), & \quad t_n\le r\le s_n \\
        \sum_{k=n+1}^\infty\vp(k)+\log\frac{1-r}{1-t_{n+1}}, & \quad s_n\le r\le t_{n+1}
        \end{array}\right.,\quad n\ge 2.
	\end{split}
	\end{equation}
Now that $\om\in\DD$ if and only if there exists $C>1$ such that $\widehat{\om}(r)\le C\widehat{\om}(1-\frac{1-r}{e})$ for all $0\le r<1$, we deduce $\om\not\in\DD$ if
	\begin{equation}\label{requirement1}
    \limsup_{n\to \infty}\frac{\sum_{k=n+1}^\infty\vp(k)}{\vp(n)}=\infty
	%\vp(n)\not\lesssim\sum_{k=n+1}^\infty\vp(k),\quad n\ge 2.
	\end{equation}
This is our first requirement for $\vp$. 

We now proceed to consider $\widehat{\widetilde{\om}}$. If $k\ge 2$ and $t_k\le t\le s_k$, then \eqref{eq:tilde} yields
	\begin{equation*}
	\begin{split}
	\widehat{\widetilde{\om}}(t)
	&=\int_t^{s_k}\frac{\widehat{\om}(r)}{1-r}\,dr
	+\sum_{n=k}^\infty\int_{s_n}^{t_{n+1}}\frac{\widehat{\om}(r)}{1-r}\,dr
	+\sum_{n=k}^\infty\int_{t_{n+1}}^{s_{n+1}}\frac{\widehat{\om}(r)}{1-r}\,dr\\
	&=\log\frac{1-t}{1-s_k}\sum_{n=k}^\infty\vp(n)
	+\sum_{n=k}^\infty\left(\log\frac{1-s_n}{1-t_{n+1}}\sum_{j=n+1}^\infty\vp(j)
	+\int_{s_n}^{t_{n+1}}\left(\log\frac{1-r}{1-t_{n+1}}\right)\frac{dr}{1-r}\right)\\
	&\quad+\sum_{n=k}^\infty\left(\log\frac{1-t_{n+1}}{1-s_{n+1}}\sum_{j=n+1}^\infty\vp(j)\right)\\
	&=\log\frac{1-t}{1-s_k}\sum_{n=k}^\infty\vp(n)\\ 
	&\quad+\sum_{n=k}^\infty\left(\log\frac{1-s_n}{1-t_{n+1}}\sum_{j=n+1}^\infty\vp(j)
    +\log\frac{1}{1-t_{n+1}}\log\frac{1-s_n}{1-t_{n+1}}
    +\int_{s_n}^{t_{n+1}}\log(1-r)\frac{dr}{1-r}\right)\\ 
		&\quad+\sum_{n=k}^\infty\left(
    \log\frac{1-t_{n+1}}{1-s_{n+1}}\sum_{j=n+1}^\infty\vp(j)\right).
	\end{split}
	\end{equation*}
Since $\log\frac{1-s_n}{1-t_{n+1}}=\vp(n)$, $\log\frac{1-t_{n+1}}{1-s_{n+1}}=1-\vp(n+1)$ and
	\begin{equation*}
	\begin{split}
	2\int_{s_n}^{t_{n+1}}\log(1-r)\frac{dr}{1-r}
	&=\left(\log\frac1{1-s_n}\right)^2-\left(\log\frac1{1-t_{n+1}}\right)^2
	=\left(n+1-\vp(n)\right)^2-\left(n+1\right)^2\\
	&=\vp(n)^2-2(n+1)\vp(n),\quad n\ge 2,
	\end{split}
	\end{equation*}
we deduce
	\begin{equation*}
	\begin{split}
	\widehat{\widetilde{\om}}(t)
	&=\log\frac{1-t}{1-s_k}\sum_{n=k}^\infty\vp(n)
	+\sum_{n=k}^\infty\left(\vp(n)\sum_{j=n+1}^\infty\vp(j)\right)
	+\frac12\sum_{n=k}^\infty\vp(n)^2\\
	&\quad+\sum_{n=k}^\infty\left(\left(1-\vp(n+1)\right)\sum_{j=n+1}^\infty\vp(j)\right),\quad t_k\le t\le s_k,\quad k\ge 2.
	\end{split}
	\end{equation*}
Recall that $\widetilde{\om}\in\DD$ if and only if $\widehat{\om}\lesssim\widehat{\widetilde{\om}}$ on $[0,1)$ by Theorem~\ref{th:1}. Obviously, for $t_k\le t\le s_k$, we have
	\begin{equation*}
	\begin{split}
	\widehat{\widetilde{\om}}(t)
	&\ge\widehat{\widetilde{\om}}(s_k)
	=\sum_{n=k}^\infty\left(\left(\vp(n)+1-\vp(n+1)\right)\sum_{j=n+1}^\infty\vp(j)\right)
	+\frac12\sum_{n=k}^\infty\vp(n)^2\\
	&\ge\frac12\sum_{n=k}^\infty\left(\sum_{j=n+1}^\infty\vp(j)\right)
	\end{split}
	\end{equation*}
because the range of $\vp$ is contained in $(0,1/2)$. In view of \eqref{eq:tilde}, our second requirement for $\vp$ is
	\begin{equation}\label{requirement2}
	\sum_{n=k}^\infty\vp(n)\lesssim\sum_{n=k}^\infty\left(\sum_{j=n+1}^\infty\vp(j)\right),\quad k\ge 2.
	\end{equation}
If $k\ge 2$ and $s_k\le t\le t_{k+1}$, then \eqref{eq:tilde} yields
	\begin{equation*}
	\begin{split}
	\widehat{\widetilde{\om}}(t)
	&=\int_t^{t_{k+1}}\frac{\widehat{\om}(r)}{1-r}\,dr+\int_{t_{k+1}}^1\frac{\widehat{\om}(r)}{1-r}=\cdots\\
	&=\log\frac{1-t}{1-t_{k+1}}\sum_{n=k+1}^\infty\vp(n)+\frac{1}{2}\left(k+1-\log\frac1{1-t}\right)^2+\frac12\sum_{n=k+1}^\infty\vp(n)^2\\
	&\quad+\left(1-\vp(k+1)\right)\sum_{n=k+1}^\infty\vp(n)
	+\sum_{n=k+1}^\infty\left(\left(\vp(n)+1-\vp(n+1)\right)\sum_{j=n+1}^\infty\vp(j)\right)
	\end{split}
	\end{equation*}
and
	\begin{equation*}
	\begin{split}
	\widehat{\om}(t)
	&=\sum_{n=k+1}^\infty\vp(n)+\log\frac{1-t}{1-t_{k+1}}.
	\end{split}
	\end{equation*}
Now that the range of $\vp$ is contained in $(0,1/2)$, we have
	$$
	\left(1-\vp(k+1)\right)\sum_{n=k+1}^\infty\vp(n)\ge\frac12\sum_{n=k+1}^\infty\vp(n)
	$$
and
	$$
	\sum_{n=k+1}^\infty\left(\left(\vp(n)+1-\vp(n+1)\right)\sum_{j=n+1}^\infty\vp(j)\right)
	\ge\frac12\sum_{n=k+1}^\infty\left(\sum_{j=n+1}^\infty\vp(j)\right).
	$$
Hence
	$$
	\widehat{\widetilde{\om}}(t)
	\ge\frac12\sum_{n=k+1}^\infty\vp(n)
	+\frac12\sum_{n=k+1}^\infty\left(\sum_{j=n+1}^\infty\vp(j)\right),\quad s_k\le t\le t_{k+1}.
	$$
Moreover,
		$$
		\log\frac{1-t}{1-t_{k+1}}\le\log\frac{1-s_k}{1-t_{k+1}}=\vp(k),\quad s_k\le t\le t_{k+1}.
		$$
Therefore, to deduce $\widehat{\om}\lesssim\widehat{\widetilde{\om}}$ on $[s_k,t_{k+1}]$, it suffices to require
	\begin{equation}\label{requirement3}
	\vp(k)\lesssim\sum_{n=k+1}^\infty\left(\sum_{j=n+1}^\infty\vp(j)\right),\quad k\ge 2.
	\end{equation}
To complete the proof, it now remains to construct $\vp:\N\setminus\{1\}\to(0,1/2)$ such that $\vp\in\ell^1$ and \eqref{requirement1}, \eqref{requirement2} and \eqref{requirement3} are satisfied. Set
	\begin{equation}\label{eq:phi}
	\begin{split}
	\vp(n)=\left\{
        \begin{array}{ll}
        n^{-3}, & n\ne 2^{2^j},\\
        \frac{\log_2\log_2 n}{n^2}+n^{-3}, & n=2^{2^j}.
        \end{array}\right.,\quad n\ge2,
	\end{split}
	\end{equation}
so that obviously $\vp\in\ell^1$. Moreover,
	$$
	\vp(2^{2^j})=\frac{j}{2^{2^{j+1}}}+\frac1{2^{3\cdot2^j}}\ge\frac{j}{2^{2^{j+1}}},\quad j\in\N,
	$$
and
	$$
	\sum_{k=2^{2^j}+1}^\infty\vp(k)
	=\sum_{k=2^{2^j}+1}^\infty\frac1{k^3}+\sum_{k=j+1}^\infty\frac{k}{2^{2^{k+1}}}
	\asymp\frac1{2^{2^{j+1}}}+\frac{j}{2^{2^{j+2}}},\quad j\in\N,
	$$
and hence the first requirement \eqref{requirement1} for $\vp$ is satisfied. We also have \eqref{requirement2} and \eqref{requirement3} because
		$$
		\sum_{n=k}^\infty\vp(n)
		\lesssim\frac{\log_2\log_2 k}{k^2}+\frac1{k^2}
		\lesssim\frac1k
		\lesssim\sum_{n=k}^\infty\left(\sum_{j=n+1}^\infty\frac1{j^3}\right)
		\le\sum_{n=k}^\infty\left(\sum_{j=n+1}^\infty\vp(j)\right),\quad k\ge 4,
		$$
and
		$$
		\vp(k)
		\lesssim\frac{\log_2\log_2 k}{k^2}+\frac1{k^3}
		\lesssim\frac1k
		\lesssim\sum_{n=k+1}^\infty\left(\sum_{j=n+1}^\infty\frac1{j^3}\right)
		\le\sum_{n=k+1}^\infty\left(\sum_{j=n+1}^\infty\vp(j)\right),\quad k\ge 4.
		$$
Therefore $\vp$ has the desired properties. Thus $\om\not\in\DD$ but $\widetilde{\om}\in\DD$.
\end{Prf}

\medskip
\begin{Prf}{\em{Theorem~\ref{pr:count}.}}
We use the construction given in the proof of \cite[Theorem~14]{PelaezRattya2019}. Let $\psi:[1,\infty)\to(0,\infty)$ an increasing unbounded function, and let $\om(r)=\om_{\psi}(r)=\sum_{n=1}^\infty\chi_{[r_{2n+1},r_{2n+2}]}(r)$, where $r_x=1-\frac{1}{2^{x\psi(x)}}$ and $\psi$ satisfies $\psi(x+1)-\psi(x)\le C/x$ for all $x\ge1$. Then
    \begin{equation}\label{20'}
    \frac{1-r_x}{1-r_{x+1}}=\frac{2^{(x+1)\psi(x+1)}}{2^{x\psi(x)}}=2^{x(\psi(x+1)-\psi(x))+\psi(x+1)}\ge 2^{\psi(x+1)}\to\infty,\quad x\to\infty,
    \end{equation}
and hence
	\begin{equation}\label{17'}
    \frac{1-r_x}{1-r_{x+1}}\asymp2^{\psi(x+1)}\asymp2^{\psi(x)},\quad x\ge1.
    \end{equation}
We know that $\om\not\in\DD$ by the proof of \cite[Theorem~14]{PelaezRattya2019}. We will show next that $\widetilde{\om}\not\in\DD$. If $r_{2n}\le s<r_{2_{2n+1}}$, then 
	\begin{equation*}
	\begin{split}
	\widehat{\om}(s)
	&=\sum_{j=n}^\infty(r_{2j+2}-r_{2j+1})
	=\sum_{j=n}^\infty\left(\frac{1}{2^{(2j+1)\psi(2j+1)}}-\frac{1}{2^{(2j+2)\psi(2j+2)}}\right)\\
	&\asymp\sum_{j=n}^\infty\frac{1}{2^{(2j+1)\psi(2j+1)}}
	\asymp1-r_{2n+1},
	\end{split}
	\end{equation*}
and if $r_{2n+1}\le s<r_{2n+2}$, then
	\begin{equation*}
	\begin{split}
	\widehat{\om}(s)
	&=\sum_{j=n+1}^\infty(r_{2j+2}-r_{2j+1})+r_{2n+2}-s
	\asymp1-s+r_{2n+2}-r_{2n+3}.
	\end{split}
	\end{equation*}
Thus
	\begin{equation*}
	\begin{split}
	\widehat{\om}(s)\asymp\left\{
        \begin{array}{ll}
        1-r_{2j+1}, & \quad r_{2j}\le s\le r_{2j+1} \\
        1-s-(r_{2j+3}-r_{2j+2}), & \quad r_{2j+1}\le s\le r_{2j+2}
        \end{array}\right.,\quad j\in\N.
	\end{split}
	\end{equation*}
If $r_{2n}\le r\le r_{2n+1}$, then \eqref{17'} yields
	\begin{equation*}
	\begin{split}
	\widehat{\widetilde{\om}}(r)
	&\asymp\sum_{j=n+1}^\infty\int_{r_{2j}}^{r_{2j+1}}\frac{1-r_{2j+1}}{1-s}\,ds
	+\int_{r}^{r_{2n+1}}\frac{1-r_{2n+1}}{1-s}\,ds
  +\sum_{j=n}^\infty\int_{r_{2j+1}}^{r_{2j+2}}\left(1-\frac{r_{2j+3}-r_{2j+2}}{1-s}\right)\,ds\\
	&=\sum_{j=n+1}^\infty\left(1-r_{2j+1}\right)\log\frac{1-r_{2j}}{1-r_{2j+1}}
	+(1-r_{2n+1})\log\frac{1-r}{1-r_{2n+1}}\\
  &\quad+\sum_{j=n}^\infty\left(r_{2j+2}-r_{2j+1}\right)
	\left(1-\frac{r_{2j+3}-r_{2j+2}}{r_{2j+2}-r_{2j+1}}\log\frac{1-r_{2j+1}}{1-r_{2j+2}}\right)\\
	&\asymp\sum_{j=n+1}^\infty\frac{\psi(2j+1)}{2^{(2j+1)\psi(2j+1)}}
	+\frac1{2^{(2n+1)\psi(2n+1)}}\log\frac{1-r}{1-r_{2n+1}}
  +\sum_{j=n}^\infty\frac1{2^{(2j+1)\psi(2j+1)}}\\
	&\asymp\frac{\psi(2n+3)}{2^{(2n+3)\psi(2n+3)}}
	+\frac1{2^{(2n+1)\psi(2n+1)}}\log\frac{1-r}{1-r_{2n+1}}
  +\frac1{2^{(2n+1)\psi(2n+1)}}\\
	&\asymp\frac1{2^{(2n+1)\psi(2n+1)}}\left(1+\log\frac{1-r}{1-r_{2n+1}}\right),	
	\end{split}
	\end{equation*}
and for $r_{2n+1}\le r\le r_{2n+2}$ we have
	\begin{equation*}
	\begin{split}
	\widehat{\widetilde{\om}}(r)
	&\asymp\int_{r}^{r_{2n+2}}\left(1-\frac{r_{2n+3}-r_{2n+2}}{1-s}\right)\,ds
	+\sum_{j=n+1}^\infty\int_{r_{2j}}^{r_{2j+1}}\frac{1-r_{2j+1}}{1-s}\,ds\\
	&\quad+\sum_{j=n+1}^\infty\int_{r_{2j+1}}^{r_{2j+2}}\left(1-\frac{r_{2j+3}-r_{2j+2}}{1-s}\right)\,ds\\
	&=(r_{2n+2}-r)-(r_{2n+3}-r_{2n+2})\log\frac{1-r}{1-r_{2n+2}}\\
	&\quad+\sum_{j=n+1}^\infty(1-r_{2j+1})\log\frac{1-r_{2j}}{1-r_{2j+1}}\\
	&\quad+\sum_{j=n+1}^\infty\left((r_{2j+2}-r_{2j+1})
	-(r_{2j+3}-r_{2j+2})\log\frac{1-r_{2j+1}}{1-r_{2j+2}}\right)\\
	&\asymp(r_{2n+2}-r)-(r_{2n+3}-r_{2n+2})\log\frac{1-r}{1-r_{2n+2}}
	+\frac{\psi(2n+3)}{2^{(2n+3)\psi(2n+3)}}.
	\end{split}
	\end{equation*}
It follows that
	$$
	\widehat{\widetilde{\om}}(r_{2n+2})\asymp\frac{\psi(2n+3)}{2^{(2n+3)\psi(2n+3)}},\quad n\in\N,
	$$
and 
	$$
	\widehat{\widetilde{\om}}(2r_{2n+2}-1)
	\asymp(1-r_{2n+2})\left(2-\log 2\right)+(1-r_{2n+3})\log2+\frac{\psi(2n+3)}{2^{(2n+3)\psi(2n+3)}},\quad n\to\infty.
	$$
Thus $\widehat{\widetilde{\om}}(r_{2n+2})/\widehat{\widetilde{\om}}(2r_{2n+2}-1)\to0$, as $n\to\infty$, and hence $\widetilde\om\not\in\DD$.

By the proof of \cite[Theorem~14]{PelaezRattya2019} we know that for a suitably taken $C_1>0$, the choice $\psi(x)=-\frac1{C_1}\log_2\vp\left(1-\frac1x\right)$ yields $\vp\in\DDD$ and $A^p\subset A^p_\om\subset A^p_{\vp}$ for all $0<p<\infty$. It remains to show that $\vp(r)=2^{-C_1\psi\left(\frac1{1-r}\right)}$ is regular, that is,
	$$
	\int_r^12^{-C_1\psi\left(\frac1{1-t}\right)}\,dt\asymp(1-r)2^{-C_1\psi(\frac1{1-r})},\quad 0\le r<1.
	$$
By a change of variable, this	is equivalent to 
	$$
	\int_x^\infty2^{-C_1\psi(y)}\frac{dy}{y^2}\asymp\frac{2^{-C_1\psi(x)}}{x},\quad 1\le x<\infty.
	$$
But since $\psi$ is increasing, we have
	\begin{equation*}
	\begin{split}
	\frac{2^{-C_1\psi(x)}}{x}
	&\ge\int_x^\infty2^{-C_1\psi(y)}\frac{dy}{y^2}
	\ge2^{-C_1\psi(2x)}\frac1{2x}
	\ge2^{-C_1\psi(x)}\frac1{2^{C_1C_2+1}x},\quad 1\le x<\infty,
	\end{split}
	\end{equation*}
for some constant $C_2=C_2(C)>0$, where $C$ is that of the hypothesis $\psi(x+1)-\psi(x)\le C/x$ for all $x\ge1$. Thus $\vp$ is regular, and the proof is complete.
\end{Prf}

\end{document}